\documentclass[10p]{amsart}
\usepackage{amssymb}
\usepackage{color}
\usepackage{epsfig}
\usepackage{url}
\usepackage{setspace}
\usepackage{color}
\theoremstyle{plain}

\newtheorem{thm}{Theorem}[section]
\newtheorem{thmast}[thm]{Theorem*}
\newtheorem{cor}[thm]{Corollary}
\newtheorem{lem}[thm]{Lemma}

\newtheorem{rem}[thm]{Remark}

\newtheorem{conj}[thm]{Conjecture}

\def\cal{\mathcal}
\def\bbb{\mathbb}
\def\op{\operatorname}
\renewcommand{\phi}{\varphi}

\newcommand{\N}{\bbb{N}}
\newcommand{\Z}{\bbb{Z}}
\newcommand{\Q}{\bbb{Q}}

\newcommand{\C}{\bbb{C}}

\begin{document}
\author{Andrew Bremner
\and Maciej Ulas}
\thanks{The research of the second author is partially supported by the grant of the Polish National Science Centre no. UMO-2012/07/E/ST1/00185}


\title[Observations concerning reducibility of quadrinomials]{Some observations concerning reducibility of quadrinomials}

\begin{abstract}
In a recent paper \cite{Jan}, Jankauskas proved some interesting
results concerning the reducibility of quadrinomials of the form
$f(4,x)$, where $f(a,x)=x^{n}+x^{m}+x^{k}+a$. He also obtained some
examples of reducible quadrinomials $f(a,x)$ with $a\in\Z$, such
that all the irreducible factors of $f(a,x)$ are of degree $\geq 3$.

In this paper we perform a more systematic approach to the problem and
ask about reducibility of $f(a,x)$ with $a\in\Q$. In particular by
computing the set of rational points on some genus two curves we
characterize in several cases all quadrinomials $f(a,x)$ with degree $\leq 6$ and
divisible by a quadratic polynomial. We also give further examples of
reducible $f(a,x)$, $a\in\Q$, such that all irreducible
factors are of degree $\geq 3$.

\bigskip



\end{abstract}

\maketitle
\section{Introduction}\label{section1}

Let $f(x)$ be a polynomial with rational coefficients. We say that
the polynomial $f$ is {\it primitive} if it is not of the form
$g(x^{l})$ for some $l\geq 2$. Throughout the paper, reducibility
of a polynomial will mean reducibility over $\Q$.

At the West Coast Number Theory conference in 2007, P.G.~Walsh
posed the following question. Let $n>m>k$ be positive integers. Does
there exist an irreducible polynomial $f(x)=x^n+x^m+x^k+4$, $\op{deg}f>17$,
such that for some integer $l>1$, the
polynomial $f(x^l)$ is reducible in $\Z[x]$?
The answer to this question was given by J.~Jankauskas in \cite{Jan}.
He proved that the only primitive quadrinomial of the form
$f(x)=x^n+x^m+x^k+4$, such that $f(x^l)$ is reducible for some
$l>1$, is the polynomial $f(x)=x^4+x^3+x^2+4$. In this case, $l=2g$, where $g\in\N_{+}$.
In particular, when $l=2$,
\begin{equation*}
f(x^{2})=(x^{4}-x^{3}+x^{2}-2x+2)(x^{4}+x^{3}+x^{2}+2x+2).
\end{equation*}

Let $a\in\Q$ and define the quadrinomial
\begin{equation*}
f_{n,m,k}(a,x)=x^{n}+x^{m}+x^{k}+a.
\end{equation*}
Walsh also asked for examples of reducible primitive $f_{n,m,k}(a,x)$ with integer
constant coeficient $a>4$,
but which have no linear or quadratic factor. He gave one such
example: $x^7+x^5+x^3+8=(x^3-x^2-x+2)(x^4+x^3+3x^2+2x+4)$. Jankauskas
in the cited paper also considered this problem using a computational
approach but did not find any additional examples. However, he
did find four examples of reducible $f_{n,m,k}(a,x)$ with
$a<-5$ such that all irreducible factors are of degree $\geq 3$.

In this paper we study Walsh's question in a more systematic way.
More precisely, we are interested in the reducibility of
quadrinomials $f_{n,m,k}(a,x)$ with $a\in\Q$. However, we drop the condition
that all irreducible factors be of degree $\geq 3$
because it is too restrictive. It is clear that if $f_{n,m,k}(a,x)$ has a
linear factor over $\Q$, say $x-p $, where $p\in \Q$, then
$a=-p^{n}-p^{m}-p^{k}$. The question about reducibility of $f_{n,m,k}(a,x)$ with $a\in\Q$,
starts to be interesting if we ask about irreducible factors of degree
$\geq 2$. In fact most of our results are related to the characterization of those
$a\in\Q$, and $n,m,k\in\N$ with $n>m>k$ and ``small" $n$, such that
$f_{n,m,k}(a,x)$ is divisible by a quadratic polynomial. Although this case
is far from general, it is highly nontrivial. Indeed, as we shall see for
fixed $n,m,k\in\N$ with $n>5$, the problem reduces to finding all rational
points on a curve of genus $\geq 2$.

\begin{rem}
{\rm
We are interested in this paper in the reducibility of $f_{n,m,k}(a,x)$, but
it is worth pointing out that it is easy to show the existence of infinitely
many irreducible such for a fixed degree $n \geq 3$, as follows.

\begin{lem}
Let $p \geq 5$ be prime. Then the quadrinomial $x^n+x^m+x^k+p$, $n>m>k\geq 1$,
is irreducible over $\Q$.
\end{lem}
\begin{proof}
Suppose $x^n+x^m+x^k+p=f_1(x)f_2(x)$ in $\Z[x]$, with $n>\deg(f_1),\deg(f_2) \geq 1$.
The constant coefficient of $f_1$ is without loss of generality $\pm p$, so
that the constant coefficient of $f_2$ is $\pm 1$. Not all roots of $f_2$ can therefore
have absolute value greater than $1$, so take $z \in \C$ as root of $f_2$ satisfying
$|z| \leq 1$. Then $p=|z^n+z^m+z^k| \leq |z|^n+|z|^m+|z|^k \leq 3$.
\end{proof}
More on (ir)reducibility of general quadrinomials can be found in \cite{FrSch}.
}
\end{rem}

\section{Reducible quadrinomials $f_{n,m,k}(a,x)$ with $\op{deg}f=4$}\label{section2}

In this section we characterize all $a\in\Q\setminus\{0\}$ and pairs of integers
$m,k\in\N$ satisfying $4>m>k$ such that $f_{4,m,k}(a,x)$ (of degree 4) is divisible
by a quadratic polynomial.

\begin{thm}\label{thm:n=4}
Let $f_{4,m,k}(a,x)=x^4+x^m+x^k+a$, where $4>m>k \geq 1$ and $a\in\Q$. Then
$f_{4,m,k}(a,x)$ has a quadratic factor precisely in the following cases:
\begin{enumerate}
\item If $(m,k)=(2,1)$ then $f_{4,2,1}(a,x)$ is divisible by $x^2+px+q$ if
and only if
\begin{equation*}
a=\frac{(p^3+p-1)(p^3+p+1)}{4p^2},\quad q=\frac{p^3+p-1}{2p}.
\end{equation*}
In this case we have
\begin{equation*}
f_{4,2,1}(a,x)=\left(x^2-px+\frac{p^3+p+1}{2p}\right)\left(x^2+px+\frac{p^3+p-1}{2p}\right).
\end{equation*}
\item If $(m,k)=(3,1)$ then $f_{4,3,1}(a,x)$ is divisible by $x^2+px+q$ if
and only if
\begin{equation*}
a=\frac{(p^3-2p^2+p+1)(p^3-p^2-1)}{(2p-1)^2},\quad q=\frac{p^3-p^2-1}{2p-1}.
\end{equation*}
In this case we have
\begin{equation*}
f_{4,3,1}(a,x)=\left(x^2-(p-1)x+\frac{p^3-2p^2+p+1}{2p-1}\right)\left(x^2+px+\frac{p^3-p^2-1}{2p-1}\right)
\end{equation*}

\item If $(m,k)=(3,2)$, then $f_{4,3,2}(a,x)$ is divisible by $x^2+px+q$ if
and only if
\begin{equation*}
a=\frac{(p-1)p(p^2-p+1)^2}{(2p-1)^2},\quad q=\frac{p(p^2-p+1)}{2p-1}.
\end{equation*}
In this case we have
\begin{equation*}
f_{4,3,2}(a,x)=\left(x^2-(p-1)x+\frac{(p-1)(p^2-p+1)}{2p-1}\right)\left(x^2+px+\frac{p(p^2-p+1)}{2p-1}\right)
\end{equation*}
\end{enumerate}
\end{thm}
\begin{proof}
The proof of the theorem is very easy and we only give a sketch of the reasoning.
We compute for fixed $(m,k)\in\{(2,1), (3,1), (3,2)\}$ the polynomial
$f_{4,m,k}(a,x)\bmod{x^2+px+q}$. This polynomial is of the form
$r(x)=A_{m,k}x+B_{m,k},$ where $A_{m,k},\; B_{m,k}\in\Z[a,p,q]$. So
if $r(x)\equiv 0$, then $f_{4,m,k}(a,x)$ is divisible by $x^2+px+q$.
The system of equations $A_{m,k}=0,\; B_{m,k}=0$ is
triangular with respect to $a,\;q$, so it can be easily solved. For
example if $(m,k)=(2,1)$, then
\begin{equation*}
A_{2,1}=-p^3-p+1+2pq,\quad B_{2,1}=a-(p^2+1)q+q^2.
\end{equation*}
and we get the expressions for $a,\;q$ displayed in the Theorem. The remaining cases
follow analogously.
\end{proof}

\begin{rem}\label{rem1}
{\rm  Recall that a polynomial is said to have reducibility type
$(n_1,n_2,...,n_k)$ if there exists a factorization of the
polynomial into irreducible polynomials of degrees $n_1$,
$n_2$,...,$n_k$. Types are ordered so that $n_1 \leq n_2 \leq ...
\leq n_k$. In a recent paper \cite{BreUlas} we study the type of
reducibility of trinomials and get several new results which
complement the results obtained by Schinzel in the series of papers
\cite{Sch1, Sch2, Sch3} (many corrections and some additional
material can be found \cite{Sch4}).  The results
contained in Theorem \ref{thm:n=4} can be used to
characterize those $a\in\Q$ such that the quadrinomial $f_{4,m,k}(a,x)$
has reducibility type $(1,1,2)$. For example,
if $(m,k)=(2,1)$ then the set of those $a\in\Q$ such that $f_{4,2,1}(a,x)$
has reducibility type $(1,1,2)$ is parameterized by the
rational points of the genus $1$ curve
\begin{equation*}
C_{2,1}:\;r^2=-p(p^3+2p+2),
\end{equation*}
which is birationally equivalent to the elliptic curve $E_{2,1}:\; y^2=x^3+x^2-x-5$ of rank 0
and trivial torsion.
Similarly, if $(m,k)=(3,1)$ we are led to the curve
\begin{equation*}
C_{3,1}:\;r^2=-(p+1)(2p-1)(2p^2-5p+5)
\end{equation*}
birationally equivalent to the elliptic curve $E_{3,1}:\; y^2+xy=x^3-x^2-5$ of
rank 0 and $\op{Tors}(E_{3,1}(\Q))=\{\cal{O},(2,-1)\}.$ The point $(2,-1)$ leads to the
 quadrinomial
\begin{equation*}
x^4+x^3+x+1=(x+1)^2(x^2-x+1).
\end{equation*}
Finally, if $(m,k)=(3,2)$ then we get the curve
\begin{equation*}
C_{3,2}:\;r^2=-(p-1)(2p-1)(2p^2-p+3),
\end{equation*}
birationally equivalent to the elliptic curve $E_{3,2}:\; y^2+xy+y=x^3-x-1$ of rank 0
and $\op{Tors}(E_{3,2}(\Q))=\{\cal{O},(1,-1)\}.$ However, the point $(1,-1)$ leads to $a=0$.

Thus if $f_{4,m,k}(a,x)$ has reducibility type $(1,1,2)$, then $(m,k)=(3,1)$ and $a=1$. }
\end{rem}

\section{Reducible quadrinomials $f_{n,m,k}(a,x)$ with $\op{deg}f=5$}\label{section3}

In this section we characterize all $a\in\Q\setminus\{0\}$ and pairs of integers
$m,k\in\N$ satisfying $5>m>k$ such that the quadrinomial $f_{5,m,k}(a,x)$
is divisible by a quadratic polynomial. However, before
stating the results we use a lemma that allows computation of the
polynomials $f_{5,m,k}(a,x)\bmod{x^2+px+q}$. More precisely we prove the
following:

\begin{lem}\label{lem1}
Let $n\in\N$. Then we have
\begin{equation*}
x^{n}\bmod{(x^2+px+q)}=A_{n}(p,q)x+B_{n}(p,q),
\end{equation*}
where $A_{0}(p,q)=0$, $A_{1}(p,q)=1$, $B_{0}(p,q)=1$, $B_{1}(p,q)=0$, and where
for $n\geq 2$:
\begin{equation*}
A_{n}(p,q)=-pA_{n-1}(p,q)-qA_{n-2}(p,q),\quad
B_{n}(p,q)=-pB_{n-1}(p,q)-qB_{n-2}(p,q).
\end{equation*}
\end{lem}
\begin{proof}
Define $x^{n}\bmod{x^2+px+q}=A_{n}(p,q)x+B_{n}(p,q)$. The
expressions for $A_{0}, A_{1}, B_{0}, B_{1}$, are clear. In order to
shorten the notation put $A_{n}=A_{n}(p,q)$ and
$B_{n}=B_{n}(p,q)$. Then
\begin{equation*}
x^{n+1}\bmod{(x^2+px+q)}=A_{n}x^2+B_{n}x=A_{n}(-px-q)+B_{n}x=A_{n+1}x+B_{n+1}.
\end{equation*}
From the last equality, $A_{n+1}=-pA_{n}+B_{n}$ and $B_{n+1}=-qA_{n}.$
Eliminating $B_{n}$ from the first equation we get the recurrence relation
for $A_{n}(p,q)$ displayed in the statement of the Lemma. Similar reasoning
gives the recurrence relation for $B_{n}(p,q)$.
\end{proof}

From the above Lemma we get the following:

\begin{cor}\label{cor1}
If $f_{n,m,k}(a,x)$ is divisible by
$x^2+px+q$ for some $a, p, q\in\Q$ with $a\neq 0$  then
\begin{equation*}
  A_{n}(p,q)+A_{m}(p,q)+A_{k}(p,q)=0\quad\mbox{and}\quad
  a=-B_{n}(p,q)-B_{m}(p,q)-B_{k}(p,q).
\end{equation*}
\end{cor}

From this Corollary it follows that divisibility of $f_{n,m,k}(a,x)$ by $x^2+px+q$
for a fixed triple of exponents $n,m,k$ is equivalent to the characterization of
rational points on the curve
\begin{equation*}
C_{n,m,k}:\;A_{n}(p,q)+A_{m}(p,q)+A_{k}(p,q)=0.
\end{equation*}
Note that from the recurrence relation for $A_{n}(p,q)$, it follows easily that
$\op{deg}_{p}A_{n}=n-1$ and
$\op{deg}_{q}A_{n}=\left\lfloor\frac{n-3}{2} \right\rfloor+1$.
Define the polynomial
\begin{equation*}
F_{n,m,k}(p,q)=A_{n}(p,q)+A_{m}(p,q)+A_{k}(p,q)
\end{equation*}
and let $\overline{F}_{n,m,k}(p,q,r)=r^{n}F_{n,m,k}(p/r,q/r)$ be
the homogenization of the polynomial $F_{n,m,k}$. We also define:
\begin{equation*}
C_{n,m,k}(\Q)=\{(p:q:r)\in\mathbb{P}^{2}(\Q):\;\overline{F}_{n,m,k}(p,q,r)=0\},
\end{equation*}
the set of rational points on the curve $C_{n,m,k}$ together with
the points at infinity.

The study of the curve $C_{n,m,k}$ or more precisely its birational
models which are hyperelliptic, and the corresponding set
$C_{n,m,k}(\Q)$ will be the main object of study in this paper.

\begin{rem}\label{rem2}
{\rm The recurrence relations satisfied by $A_{n}(p,q)$ and
$B_{n}(p,q)$ allow us to get Binet type formula for these
polynomials. Indeed, using the standard method and initial
conditions for the sequence $A_{n}(p,q)$ we get that
\begin{equation}
\label{Binet}
A_{n}(p,q)=\frac{1}{\sqrt{p^2-4q}}\left(\left(\frac{\sqrt{p^2-4q}-p}{2}\right)^{n}-\left(\frac{-\sqrt{p^2-4q}-p}{2}\right)^{n}\right).
\end{equation}
A similar result can be given for $B_{n}(p,q)$ because $B_{n}(p,q)=-qA_{n-1}(p,q)$. }
\end{rem}

Now we are ready to prove the following:

\begin{thm}\label{thm:n=5}
Let $f_{5,m,k}(a,x)=x^{5}+x^{m}+x^{k}+a$ with $m,k\in\N$ and $5>m>k \geq 1$. Then the
following holds:
\begin{enumerate}
\item If $(m,k)\in\{(2,1),\;(4,1),\;(4,2),\;(4,3)\}$ then the set $C_{5,m,k}(\Q)$ is finite.
The only reducible $f_{5,m,k}(a,x)$ with quadratic factor is given by $(m,k)=(4,3)$ and $a=-1$
with factor $x^2+1$.

\item If $(m,k)\in \{(3,1),\;(3,2)\}$ then the set $C_{5,m,k}(\Q)$ is
infinite and is parameterized by the rational points on a certain elliptic curve of positive rank.

\end{enumerate}
\end{thm}
\begin{proof}
We perform case by case analysis.

Case $(m,k)=(2,1)$. Solving the equation
$F_{5,2,1}(p,q)=0$ in rational numbers $p,q$ is equivalent to the
study of rational points on the hyperelliptic quartic curve
\begin{equation*}
H_{5,2,1}:\;s^2=5p^4+4p-4,
\end{equation*}
where $s=\pm(2q-3p^2)$.
Setting
\[ (x,y)=((-10p^3+20p^2+1)/s^2, (25p^6-100p^5-50p^3+100p^2-80p-2 - s^3)/(2s^3)), \]
there results $E_{5,2,1}: y^2+y=x^3+5x+1$, an elliptic curve of rank 0 and trivial torsion.
It follows that the set $H_{5,2,1}(\Q)$ is empty.

\bigskip

Case $(m,k)=(3,1)$. Solving the equation $F_{5,3,1}(p,q)=0$ in rational numbers $p,q$ is
equivalent to the study of rational points on the hyperelliptic quartic curve
\begin{equation*}
H_{5,3,1}:\;s^2=(p^2+1)(5p^2-3),
\end{equation*}
where $s=\pm(2q-3p^2-1)$. Using the point $(p,s)=(1,2)$ as a zero point,
then $H_{5,3,1}$ is birationally equivalent to the elliptic curve
\begin{equation*}
E_{5,3,1}:\;Y^2=X^3-X^2+4X
\end{equation*}
under the mapping
\begin{equation*}
(p,s)=\phi(X,Y)=\left(\frac{X+4-Y}{3X-Y-4},\frac{2(X^3-12X^2+4X+16Y-16)}{(3X-Y-4)^{2}}\right).
\end{equation*}
We have that $\op{Tors}(E_{5,3,1}(\Q))\cong\{\cal{O},(0,0)\}$
and the rank of $E_{5,3,1}$ is one, with generator $P=(1,2)$.
Our reasoning shows that if $f_{5,3,1}(a,x)$ is divisible by $x^2+px+q$ then
$p=\phi_{1}(X,Y)$, where $\phi_{1}$ is the first coordinate of the
function $\phi$, and $q$'s can be computed (if needed) from the
relation $\phi_{2}(X,Y)=\pm(2q-1-3\phi_{1}(X,Y)^2)$. The corresponding
$a\in\Q$ can be computed from the expression given in Corollary
\ref{cor1}. In particular there are infinitely many such $a$
corresponding to the integer multiples of $P$. For example, the
point $-P$ corresponds to $p=-1, q=3$ and $a=-12$, which leads to the
factorization
\begin{equation*}
x^{5}+x^{3}+x-12=(x^2-x+3)(x^3+x^2-x-4).
\end{equation*}

\bigskip

Case $(m,k)=(4,1)$. Solving the equation
$F_{5,4,1}(p,q)=0$ in rational numbers $p, q$ is equivalent to the
study of rational points on the hyperelliptic quartic curve
\begin{equation*}
H_{5,4,1}:\;s^2=5p^4-8p^3+4p^2-4,
\end{equation*}
where $s=\pm(2q-3p^2+2p)$. However, the curve $H_{5,4,1}$ is not
locally solvable at the prime 2, and thus the set $H_{5,4,1}(\Q)$ is
empty.

\bigskip

Case $(m,k)=(3,2)$. Solving the equation
$F_{5,3,2}(p,q)=0$ in rational numbers $p,q$ is equivalent to the
study of rational points on the hyperelliptic quartic curve
\begin{equation*}
H_{5,3,2}:\;s^2=5p^4+2p^2+4p+1,
\end{equation*}
where $s=\pm(2q-3p^2-1)$. Using the point $(p,s)=(0,1)$ as a point
at infinity we find that $H_{5,3,2}$ is birationally equivalent to
the elliptic curve
\begin{equation*}
E_{5,3,2}:\;Y^2+Y=X^3-X^2-X+1
\end{equation*}
under the mapping
\begin{equation*}
(p,s)=\phi(X,Y)=\left(\frac{X-1}{Y-X+1},\frac{Y+X^3-3X^2+3X-1}{(Y-X+1)^{2}}\right).
\end{equation*}
The torsion subgroup of $E_{5,3,2}(\Q)$ is trivial and the rank of
the curve is one, with generator $P=(1,0)$. The point $P$ corresponds to $a=1$,
and the point $2P$, for example, leads to $a=6$ and $a=-363$. The latter
gives the factorization
\begin{equation*}
x^5+x^3+x^2-363=(x^2-2x+11)(x^3+2x^2-6x-33).
\end{equation*}

\bigskip

Case $(m,k)=(4,2)$. Solving the equation
$F_{5,4,2}(p,q)=0$ in rational numbers $p, q$ is equivalent to the
study of rational points on the hyperelliptic quartic curve
\begin{equation*}
H_{5,4,2}:\;s^2=p(5p^3-8p^2+4p+4),
\end{equation*}
where $s=\pm(2q-3p^2+2p)$. Using the point $(p,s)=(0,0)$ as a point
at infinity we find that $H_{5,4,2}$ is birationally equivalent to
the elliptic curve
\begin{equation*}
E_{5,4,2}:\;Y^2+Y=X^3+X^2-2X+1.
\end{equation*}
which has trivial torsion and rank 0. It follows that $H_{5,4,2}(\Q)=\{(0,0)\}$.
The point $(0,0)$ leads to $a=0$, with no corresponding quadrinomial.

\bigskip

Case $(m,k)=(4,3)$. Solving the equation
$F_{5,4,3}(p,q)=0$ in rational numbers $p, q$ is equivalent to the
study of rational points on the hyperelliptic quartic curve
\begin{equation*}
H_{5,4,3}:\;s^2=(p-1)(5p^3-3p^2+3p-1),
\end{equation*}
where $s=\pm(2q-3p^2+2p-1)$. Using the point $(p,s)=(1,0)$ as a
point at infinity we find that $H_{3,4,5}$ is birationally
equivalent to the elliptic curve
\begin{equation*}
E_{5,4,3}:\;Y^2+Y=X^3
\end{equation*}
with rank 0 and $\op{Tors}(E_{5,4,3}(\Q))=\{\cal{O}, (0,0), (0,-1)\}$.
It follows that $H_{5,4,3}(\Q)=\{(0,-1), (1,0)\}$. The point $(0,-1)$ leads to
$a=-1$ and the point $(1,0)$ leads to $a=0$.

Tying now all the results together we get the statement of the Theorem.
\end{proof}

\section{Reducible quadrinomials $f_{n,m,k}(a,x)$ with $\op{deg}f=6$}\label{section4}

In the previous section the problem of finding
reducible quadrinomials $f_{n,m,k}(a,x)$ with $\op{deg}f=5$ and $a\in\Q$
resulted in the examination of the set of rational
points on several curves of genus one. In this section we examine
the divisibility of a sextic $f_{6,m,k}(a,x)$ by the quadratic polynomial $x^2+px+q$,
which results in the investigation of several curves of genus two.
Computing all the rational points on any specific curve of genus 2
is unfortunately still very much an open problem. In some cases,
in particular if the Jacobian of the curve has rank at most 1, then
Chabauty arguments can work. Otherwise elliptic Chabauty techniques
can sometimes apply, occasionally with a great deal of associated work,
allowing all rational points to be determined. But in other cases, the methods
do not fall to computation. If a curve takes the form $y^2=g(x)$ for an
irreducible quintic polynomial $g$, and has Jacobian of rank at least 2, then
the only approach known to us is to factor over the quintic field defined by a root
of $g$, and try to apply elliptic Chabauty techniques. But the associated
arithmetic often leaves a machine churning.  In what follows, some results are absolute,
when we have been able to determine all the rational points on a given curve;
and some results are conjectural, when we have been unable to prove
a given set of points (usually found by search up to a height bound of $10^6$)
is complete. We introduce an Asterisked Theorem, where cases without asterisk
are absolute, and cases with asterisk represent instances where the reader will know
that the result depends upon a set of listed points on some curve of genus 2 being complete.

\begin{thmast}\label{thm:n=6}
Let $f_{6,m,k}(a,x)=x^{6}+x^{m}+x^{k}+a$, $m,k\in\N$ with $6>m>k \geq 1$ and suppose
that $f_{6,m,k}(a,x)$ is divisible by $x^2+px+q$. Then the following hold, with quadratic factors listed respectively:
\begin{enumerate}
\item* If $(m,k)=(2,1)$ then $a\in\left\{ -\frac{5795}{1728},-\frac{3655}{1728}, -\frac{1115}{1728}, \frac{10}{27}, \frac{51}{64}\right\}$, with factors
\[ x^2-\frac{3x}{2}+\frac{19}{12}, \quad x^2-\frac{3x}{2}+\frac{17}{12}, \quad x^2+\frac{x}{2}-\frac{5}{12}, \quad x^2+x+\frac{1}{3}, \quad x^2+\frac{x}{2}+\frac{3}{4}. \]
\item* If $(m,k)=(3,1)$ then $a\in\left\{-\frac{48814883}{216000}, -\frac{2899}{1728}, -\frac{21}{64}, \frac{639}{64}\right\}$, with factors
\[ x^2-\frac{5x}{2}+\frac{373}{60}, \quad x^2+\frac{x}{2}-\frac{13}{12}, \quad x^2+\frac{x}{2} +\frac{3}{4}, \quad x^2-\frac{5x}{2}+\frac{9}{4}. \]
\item
If $(m,k)=(4,1)$ then $a\in\left\{-\frac{34881}{8}, \frac{1441}{8}\right\}$,
with factors
\[ x^2-4x+\frac{33}{2}, \quad x^2-4x+\frac{11}{2}. \]

\item* If $(m,k)=(5,1)$ then $a\in\left\{-\frac{2869795813}{512}, -21, -\frac{3}{8}, 1, 206682\right\}$, with factors
\[ x^2-13x+\frac{1403}{8}, \quad x^2+2x+3, \quad x^2+x-\frac{1}{2}, \quad (x+1)^2, \quad x^2-13x+57. \]
\item
If $(m,k)=(3,2)$ then $a\in\left\{-\frac{474281}{46656}, \frac{3}{64}\right\}$, with factors
\[ x^2+\frac{3x}{2}+\frac{73}{36}, \quad x^2+\frac{3x}{2}+\frac{3}{4}. \]

\item
If $(m,k)=(4,2)$ then $a=q^3-q^2+q$, $q\in\Q$, with factor $x^2+q$.
\item* If $(m,k)=(5,2)$ then $a\in\left\{-1, -\frac{9}{64}, -\frac{2914724300237}{21596433224192},
  \frac{1983235848957}{14467923038863}\right\}$, with factors
\[ x^2-x+1, \quad x^2-x+\frac{3}{4}, \quad x^2+\frac{x}{59}-\frac{3709}{27848}, \quad x^2+\frac{x}{59}+\frac{3267}{24367}. \]

\item
If $(m,k)=(4,3)$ then $a\in\left\{-\frac{729}{8}, \frac{88}{27} \right\}$, with factors
\[ x^2+2x+\frac{9}{2}, \quad x^2+2x+\frac{4}{3}. \]

\item* If $(m,k)=(5,3)$ then $a\in\left\{-27, -\frac{27}{8}, \frac{16651236563}{1313651921408},
  \frac{222528399633}{1313651921408}, \frac{24}{125}, \frac{3}{8}, 1 \right\}$, with factors
\[ x^2+2x+3, \; x^2-x+\frac{3}{2}, \; x^2-\frac{9x}{37}+\frac{611}{10952}, \; x^2-\frac{9x}{37}+\frac{7209}{10952}, \; x^2+2x+\frac{4}{5}, \; x^2-x+\frac{1}{2}, \; x^2+1. \]

\item
If $(m,k)=(5,4)$ then the set, say $\cal{A}$, of those $a\in\Q$ such that $f(a,x)$ is
divisible by $x^2+px+q$ is infinite. More precisely the set $\cal{A}$ is parameterized by
the rational points on the rank one elliptic curve
\begin{equation*}
E_{6,5,4}:\;Y^2+XY+Y=X^3-X^2-2X,
\end{equation*}
with $\op{Tors}(E_{6,5,4}(\Q))=\{\cal{O},(-1,0)\}$, and where a generator of infinite order
is $P=(0,0)$.

\end{enumerate}
\end{thmast}

\begin{proof}
As before, we perform a case by case analysis.

\bigskip

Case $(m,k)=(2,1)$. Solving the equation $F_{6,2,1}(p,q)=0$ in rational numbers $p,q$, is
equivalent to the study of rational points on the hyperelliptic sextic curve of genus 2
\begin{equation*}
H_{6,2,1}:\;s^2=p(p^5-3p+3),
\end{equation*}
where $s=\pm(3pq-2p^3)$. On the curve $H_{6,2,1}$ we have rational
points with $p\in\left\{-\frac{3}{2}, 0, \frac{1}{2}, 1\right\}$ and
these numbers lead to the $a$'s displayed in the statement of the
theorem. The rank of the Jacobian variety of $H_{6,2,1}$ is equal to 2, and traditional
Chabauty arguments are therefore unavailable.
One can factor over the quintic field $\Q(\theta)$, $\theta^5-3\theta+3=0$, to
obtain
\[ 3(3P+(\theta^4-3)) (3P^4-\theta^4 P^3-\theta^3 P^2-\theta^2 P-\theta) = \square \]
with $P=1/p$. This allows deduction of elliptic quartics of the type
\[ 3P^4-\theta^4 P^3-\theta^3 P^2-\theta^2 P-\theta = \delta \square \]
for a finite number of $\delta \in \Q(\theta)$, and there is a possibility that
an approach using elliptic Chabauty methods will work. However, we were unsuccessful
in getting this method to work, and have been unable to find explicitly all points on $H_{6,2,1}$.

\bigskip

Case $(m,k)=(3,1)$. Solving the equation
$F_{6,3,1}(p,q)=0$ in rational numbers $p, q$ is equivalent to the
study of rational points on the hyperelliptic sextic curve
\begin{equation*}
H_{6,3,1}:\;s^2=4p^6+4p^3+12p+1,
\end{equation*}
where $s=\pm(6pq-4p^3+1)$. On the curve $H_{6,3,1}$ we have rational
points with $p\in\left\{-\frac{5}{2}, 0,\frac{1}{2}\right\}$ and these
numbers lead to the $a$'s displayed in the statement of the theorem.
The rank of the Jacobian variety is 2, so traditional Chabauty techniques are
not applicable, and we are unable to determine explicitly all the rational points
on $H_{6,3,1}$.

\bigskip

Case $(m,k)=(4,1)$. Solving the equation
$F_{6,4,1}(p,q)=0$ in rational numbers $p, q$ is equivalent to the
study of rational points on the hyperelliptic sextic curve
\begin{equation*}
H_{6,4,1}:\;s^2=p(p+1)(p^4-p^3+2p^2-2p+3),
\end{equation*}
where $s=\pm(3pq-2p^3-p)$. On the curve $H_{6,4,1}$ we have rational
points with $p=-4,-1,0$ and these lead to the $a$'s displayed in
the statement of the theorem.
The rank of the Jacobian variety of $H_{6,4,1}$ is 1, and Chabauty's method
as implemented in Magma~\cite{Mag} is able to obtain the explicit list of rational points
on $H_{6,4,1}$, precisely the three above.

\bigskip

Case $(m,k)=(5,1)$. Solving the equation
$F_{6,5,1}(p,q)=0$ in rational numbers $p, q$ is equivalent to the
study of rational points on the hyperelliptic sextic curve
\begin{equation*}
H_{6,5,1}:\;s^2=4p^6-8p^5+5p^4+12p-4,
\end{equation*}
where $s=\pm(2(-3p+1)q+4p^3-3p^2)$. On the curve $H_{6,5,1}$ we have
rational points with $p\in\left\{-1, \frac{1}{3}, 1, 2, -13\right\}$ and these
numbers lead to the $a$'s displayed in the statement of the theorem.
The rank of the Jacobian variety is 3, and we have not attempted to find
explicitly all the rational points.

\bigskip

Case $(m,k)=(3,2)$. Solving the equation
$F_{6,3,2}(p,q)=0$ in rational numbers $p, q$ is equivalent to the
study of rational points on the hyperelliptic sextic curve
\begin{equation*}
H_{6,3,2}:\;s^2=4p^6+4p^3-12p^2+1,
\end{equation*}
where $s=\pm(6pq-4p^3+1)$. On the curve $H_{6,3,2}$ we have rational
points with $p\in\left\{0, \frac{3}{2}\right\}$ and these numbers
lead to the $a$'s displayed in the statement of the theorem.
The rank of the Jacobian variety is 1, and Magma's Chabauty routines
determine the full set of rational points as precisely those given above.

\bigskip

Case $(m,k)=(4,2)$. Solving the equation
$F_{6,4,2}(p,q)=0$ in rational numbers $p,q$ leads to $p^2(p^2-1)(p^2+2)=\square$,
so leads either to $p=0$, or to study of rational points on the hyperelliptic
quartic curve
\begin{equation*}
H_{6,4,2}:\;s^2=(p+1)(p-1)(p^2+2),
\end{equation*}
where $s=\pm (3q-2p^2-1)$.
In this latter case, the curve is birationally equivalent to the elliptic curve
\begin{equation*}
E_{6,4,2}:\;Y^2=X^3+X^2+8X+8.
\end{equation*}
of rank 0 and $\op{Tors}(E_{6,4,2}(\Q))=\{\cal{O},(2,-6),(2,6),(-1,0)\}$.
The torsion points all correspond to $a=0$. \\

In the first case, then $p=0$, $a=q^3-q^2+q$, and the cubic $x^3+x^2+x+a$
has the  rational root $x=-q$.

\bigskip

Case $(m,k)=(5,2)$. Solving the equation
$F_{6,5,2}(p,q)=0$ in rational numbers $p,q$ is equivalent to the
study of rational points on the hyperelliptic sextic curve
\begin{equation*}
H_{6,5,2}:\;s^2=p(4p^5-8p^4+5p^3-12p+4),
\end{equation*}
where $s=\pm(2(-3p+1)q+4p^3-3p^2)$. On the curve $H_{6,5,2}$ we have
rational points with $p\in\left\{-1, 0, \frac{1}{3}, \frac{1}{59}\right\}$ and
these numbers lead to the $a$'s displayed in the statement of the
theorem.
The rank of the Jacobian variety is 3, and traditional Chabauty arguments
do not apply.  As in the case $(m,k)=(2,1)$ we attempted an attack
using elliptic Chabauty methods, but were unsuccessful; so
have been unable to find explicitly all points on $H_{6,5,2}$.

\bigskip

Case $(m,k)=(4,3)$. Solving the equation
$F_{6,4,3}(p,q)=0$ in rational numbers $p,q$ is equivalent to the
study of rational points on the hyperelliptic sextic curve
\begin{equation*}
H_{6,4,3}:\;s^2=4p^6+4p^4+4p^3+4p^2-4p+1,
\end{equation*}
where $s=\pm(6pq-4p^3-2p+1)$. On the curve $H_{6,4,3}$ we have
rational points with $p=0, 2$ and this number leads to the $a$'s
displayed in the statement of the theorem.
The rank of the Jacobian variety is 1, and Magma's Chabauty routines
is successful in finding all the rational points, which are precisely as above.

\bigskip

Case $(m,k)=(5,3)$. Solving the equation
$F_{6,5,3}(p,q)=0$ in rational numbers $p,q$ is equivalent to the
study of rational points on the hyperelliptic sextic curve
\begin{equation*}
H_{6,5,3}:\;s^2=4p^6-8p^5+5p^4+4p^3+2p^2+1,
\end{equation*}
where $s=\pm(2(-3p+1)q+4p^3-3p^2-1)$. On the curve $H_{6,5,3}$ we
have rational points with $p\in\left\{-1, -\frac{9}{37}, 0, \frac{1}{3}, 2, \right\}$
and these numbers lead to the $a$'s displayed in the statement of the theorem.
The rank of the Jacobian variety is 2, and we are unable to determine explicitly
all the rational points.

\bigskip

Case $(m,k)=(5,4)$. Solving the equation
$F_{6,5,4}(p,q)=0$ in rational numbers $p,q$ leads to $p^2(4p^4-8p^3+9p^2-8p+4)=\square$,
so leads to $p=0$, which in turn leads to $q=0$, or to the study of rational points on the
hyperelliptic quartic curve
\begin{equation*}
H_{6,5,4}:\;s^2=4p^4-8p^3+9p^2-8p+4,
\end{equation*}
where $s=\pm(2(-3p+1)q/p+4p^2-3p+2)$.
Taking the point $(0,2)$ as a point at infinity, we get that
$H_{6,5,4}$ is birationally equivalent to the elliptic curve
\begin{equation*}
E_{6,5,4}:\;Y^2+XY+Y=X^3-X^2-2X
\end{equation*}
We have that $\op{Tors}(E_{6,5,4}(\Q))=\{\cal{O}, (-1,0)\}$, and that the rank of
$E_{6,5,4}(\Q)$ is 1, with generator $P=(0,0)$. For example, the point $2P=(3,2)$ leads to $a=18$
and $a=-\frac{194481}{512}$. We have the following factorization in the case $a=18$:
\begin{equation*}
x^6+x^5+x^4+18=(x^2+3x+3)(x^4-2x^3+4x^2-6x+6).
\end{equation*}

Tying now all these results together we get the statement of the theorem.
\end{proof}

\begin{rem}
{\rm
There are some instances where $f_{6,m,k}(a,x)$ can split as the product of two
cubics, for example when $(m,k)=(5,3)$:
\begin{align*}
(x^3-x^2+x-\frac{1}{3})(x^3+2x^2+x+\frac{1}{3})= & x^6+x^5+x^3-\frac{1}{9}, \\
(x^3+ \frac{1}{2} x+\frac{1}{4})(x^3+x^2-\frac{1}{2} x+\frac{1}{4})= & x^6+x^5+x^3+\frac{1}{16}, \\
(x^3-\frac{1}{4} x^2-\frac{3}{32} x+\frac{117}{512})(x^3+\frac{5}{4} x^2+\frac{17}{32} x+\frac{507}{512}) = & x^6+x^5+x^3+\frac{59319}{262144}.
\end{align*}
For all $(m,k)$ except ($4,2)$ and $(5,4)$, the variety that parameterizes such
examples is a non-hyperelliptic curve of genus 4, and so other than searching for
small points (resulting in the above examples) we have not carried investigation further
here. In the remaining two cases, we can describe precisely when the sextic splits as
the product of two cubics. \\
\begin{thm}
The quadrinomial $x^6+x^4+x^2+a$ is the product of two cubics precisely when
$a=-\frac{(u-1)^2(u+1)^2(3+u^2)^2}{64u^2}$, for $u \in \Q\setminus\{0\}$.
The quadrinomial $x^6+x^5+x^4+a$, $a \neq 0$, cannot split as the product of two cubics.
\end{thm}
\begin{proof}
Case $(m,k)=(4,2)$. We have $f_{6,4,2}(a,x)=g(x^2)$ for a cubic polynomial $g$, and from Lemma 29 of \cite{Sch1}
\begin{equation*}
f_{6,4,2}(a,x)=-h(x)h(-x),
\end{equation*}
where $h(x)=x^3+ux^2+vx+w$, say. Comparing coefficients gives the (triangular) system
of equations
\begin{equation*}
1+u^2 - 2v=0,\quad 1-v^2+2uw=0,\quad a+w^2=0;
\end{equation*}
solving this system with respect to $v,w$ and $a$,
\begin{equation*}
v=\frac{u^2+1}{2}, \qquad w=\frac{(u-1)(u+1)(u^2+3)}{8u}, \qquad a=-\frac{(u-1)^2(u+1)^2(3+u^2)^2}{64u^2}.
\end{equation*}
Then $f_{6,4,2}(a,x)=-h(x)h(-x)$, where
\begin{equation*}
h(x)=x^3+u x^2+\frac{u^2+1}{2}x+\frac{(u-1)(u+1)(u^2+3)}{8u}.
\end{equation*}

\bigskip

\noindent
Case $(m,k)=(5,4)$. Suppose that
\[ x^6+x^5+x^4+a = (x^3+p x^2+q x+r)(x^3+s x^2+t x+u). \]
Comparing coefficients,
\[ -1+p+s=0, \; -1+q+p s+t=0, \; r+q s+p t+u=0, \; r s+q t+p u=0, \; r t+q u=0. \]
Eliminating $r,s,t,u$,
\begin{equation}
\label{deg6}
-p^2(p^2-p+1)^2 + (p^2-p+1)(3p^2-p+1)q-(3p^2-p+2)q^2+2q^3 = 0,
\end{equation}
a curve of genus 2. Under the mapping
\begin{align*}
& (X,Y)=( (p^2-p+1-q)/q, \\
& (-2(p-1)p(p^2-p+1)^2+(p^2-p+1)(5p^2-5p+1)q-(5p^2-3p+1)q^2+2q^3)/q^3 )
\end{align*}
we obtain a hyperelliptic model
\[ C: \; Y^2 = X(4X^4+X^3+6X^2+X+4). \]
Now the Jacobian of $C$ has rank 1, and Magma's Chabauty routines determine that
the only finite rational points on $C$ are $(X,\pm Y)=(0,0),(1,4)$. This in turn
gives $(p,q)=(0,0),(1,1),(\frac{1}{2},\frac{3}{8})$ as the complete set of finite
rational points on (\ref{deg6}), leading only to $a=0$.
\end{proof}

}
\end{rem}

\section{Reducible quadrinomials $f_{n,m,k}(a,x)$ with $\op{deg}f\geq 7$}\label{section5}

The problem of divisibility of $f(a,x)=x^{n}+x^{n-m}+x^{n-2m}+a$ by $x^2+px+q$
for some $p, q, a\in\Q$ can be reduced to the study of rational points on certain
hyperelliptic curves. More precisely:
\begin{thmast}
Let $f(a,x)=x^n+x^{n-m}+x^{n-2m}+a,\;n>2m\geq 3$ and suppose that
$f(a,x)$ is divisible by $x^2+px+q$ for some $a, p, q\in\Q$ with
$a\neq0$. Then there exists $t \in \Q$ such that $q=tp^2$
and
\begin{equation*}
A_{n-m}(1,t)^{2}-4A_{n-2m}(1,t)A_{n}(1,t)=\square.
\end{equation*}
In particular if $m=1$, we have the following results:
\begin{enumerate}
\item If $n=7$ then $a\in\{-2, 1\}$ with respective factorizations
\begin{align*}
& x^7+x^6+x^5-2=(x^2-x+1)(x^5+2x^4+2x^3-2x-2),\\
& x^7+x^6+x^5+1=(x+1)(x^2+1)(x^4-x+1).
\end{align*}
\item* If $n=8$ then
\begin{equation*}
a\in
\left\{ \frac{2}{81},
\frac{419928937515451125000}{53933980683177204481},
-\frac{13149874643832399539673}{1262324516855259464728576}\right\}.
\end{equation*}
with respective factors $x^2+x+\frac{1}{3}$, $x^2+\frac{90}{71} x+\frac{174150}{85697}$, and
$x^2+\frac{323}{728} x+\frac{263891}{1059968}$.

\end{enumerate}
\end{thmast}
\begin{proof}
It is clear that for any given pair $p, q$ of rational numbers one
can find a rational number $t$ such that $q=tp^2$. Now,
recall the formula for $A_{n}(p,q)$:
\begin{equation*}
A_{n}(p,q)=\frac{1}{\sqrt{p^2-4q}}\left(\left(\frac{\sqrt{p^2-4q}-p}{2}\right)^{n}-\left(\frac{-\sqrt{p^2-4q}-p}{2}\right)^{n}\right),
\end{equation*}
and note that $A_{n}(p,q)=A_{n}(p,p^2t)=p^{n-1}A_{n}(1,t)$. Thus we have that
\begin{equation*}
A_{n}(p,q)+A_{n-m}(p,q)+A_{n-2m}(p,q)=p^{n-2m-1}(p^{2m}A_{n}(1,t)+p^{m}A_{n-1}(1,t)+A_{n-2}(1,t)).
\end{equation*}
In order to find all rational solutions of the equation
$A_{n}(p,q)+A_{n-m}(p,q)+A_{n-2m}(p,q)=0$ it is enough now to characterize the set of rational
points on the hyperelliptic curve
\begin{equation*}
H_{n,n-m,n-2m}:\;s^2=A_{n-m}(1,t)^{2}-4A_{n-2m}(1,t)A_{n}(1,t),
\end{equation*}
where $s=\pm(2A_{n}(1,t)p+A_{n-m}(1,t))$.

Suppose $m=1$, and consider those $n$ such that the curve $H_{n,n-1,n-2}$ has genus two;
it is straightforward to check that $n=7$ and $n=8$ are the only cases.

Case $n=7$. Here, the problem of divisibility of $f(a,x)$ by $x^2+px+q$ is equivalent
to the study of rational points on the genus 2 hyperelliptic curve:
\begin{equation*}
H_{7,6,5}:\;s^2=4t^5-27t^4+72t^3-66t^2+24t-3,
\end{equation*}
where $s=\pm(2(t^3-6t^2+5t-1)p+(t-1)(3t-1))$.
The rank of the Jacobian variety is 1, and Magma's Chabauty routines determine
that the only finite rational points on $H_{7,6,5}$ are $(t,\pm s)=(1,2)$.
These return $a=-2$; the infinite point returns $a=1$.

\bigskip

Case $n=8$. Here, the problem of divisibility of
$f(a,x)$ by $x^2+px+q$ is equivalent to the study of the rational points on the
genus 2 hyperelliptic curve:
\begin{equation*}
H_{8,7,6}:\;s^2=t^6+36t^5-138t^4+186t^3-111t^2+30t-3,
\end{equation*}
where $s=\pm2((2t-1)(2t^2-4t+1)p-t^3+6t^2-5t+1)$.
There are points
at $t=\frac{1}{3}, \frac{1}{2}, 1, \frac{43}{34}$.
The rank of the Jacobian variety is 2, so standard Chabauty arguments do not apply.
The sextic factors over $\Q(\sqrt{3})$:
\[ s^2 = (t^3+18 t^2-15 t+3)^2 - 3 (12 t^2-10 t+2)^2, \]
so that over the sextic field defined by a root of the right hand side of $H_{8,7,6}$,
we obtain a factorization of type $d_1(t) d_2(t) d_3(t)= \square$, with $\op{deg}d_i=i$.
One method for resolution reduces to considering elliptic quartics of type $d_1(t) d_3(t)=\delta \square$,
for a finite number of $\delta$ over a sextic field. But again, we were unsuccessful
in completing this approach.
Consequently, we are unable to determine explicitly all the rational points on
$H_{8,7,6}$.

\end{proof}

Although we have not proved the following, it seems plausible:
\begin{conj}
Let $m,n\in\N_{+}$ with $n>2m$ and define the polynomial
\begin{equation*}
F_{m,n}(t)=A_{n-m}(1,t)^{2}-4A_{n-2m}(1,t)A_{n}(1,t).
\end{equation*}
Then $F_{m,n}$ has no multiple roots.
\end{conj}



\begin{thm}
Let $f(a,x)=x^8+x^m+x^k+a$, with $k, m$ even. Suppose that
$f(a,x)$ is reducible and $x^4+x^{\frac{m}{2}}+x^{\frac{k}{2}}+a$ is
irreducible. We have the following results:
\begin{enumerate}
\item If $(m,k)=(4,2)$, then $a\in\left\{\frac{1}{4},
\frac{625}{4}\right\}$ with respective factorizations
\begin{align*}
& x^8+x^4+x^2+\frac{1}{4}=\left( x^4-2x^3+2x^2-x+\frac{1}{2}\right)\left(x^4+2x^3+2x^2+x+\frac{1}{2}\right),\\
& x^8+x^4+x^2+\frac{625}{4}=\left(x^4-2x^3+2x^2-7x+\frac{25}{2}\right)\left(x^4+2x^3+2x^2+7x+\frac{25}{2}\right).
\end{align*}

\item If $(m,k)=(6,2)$, then there does not exist $a\in\Q$ satisfying the required properties.
\item If $(m,k)=(6,4)$, then $a=4$ with the factorization
\begin{equation*}
x^8+x^6+x^4+4=(x^4-x^3+x^2-2x+2)(x^4+x^3+x^2+2x+2).
\end{equation*}
\end{enumerate}
\end{thm}

\begin{proof}
Case $(m,k)=(4,2)$. By Lemma 29 in \cite{Sch1} we know that $f(a,x)=h(x)h(-x)$, where, say,
\begin{equation*}
h(x)=x^4+px^3+qx^2+rx+s.
\end{equation*}
Comparing now the coefficients in the equality $f(a,x)=h(x)h(-x)$ we
get the following system of equations
\begin{equation*}
  p^2-2q=0,   \; 1-q^2+2pr-2s=0, \; 1+r^2-2qs=0,\; a-s^2=0.
\end{equation*}
Solving the first, the second and the fourth equation with the
respect to $a, q, r$ we get
\begin{equation*}
q=\frac{p^2}{2},\quad r=\frac{-4+p^4+8s}{8p},\quad a=s^2.
\end{equation*}
We are thus left with finding rational points on the genus two curve
given by
\begin{equation*}
H_{8,4,2}:\;v^2=2(p^6+4p^2-8),
\end{equation*}
where $v=\pm\left(\frac{8s-3p^4-4}{2p}\right)$.
The rank of the Jacobian variety is 2, so standard Chabauty techniques are unavailable
to us.  However, we can work over the cubic number field $L$ defined by $\theta^3+4\theta-8=0$,
when the equation of the curve takes the form
\[ v^2 = 2(p^2-\theta)(p^4+\theta p^2+\theta^2+4). \]
Factoring over $L$ results in two cases to consider, one of which is locally unsolvable,
and the other of which is amenable to an argument using elliptic Chabauty techniques.
The complete set of rational points is $(\pm p,\pm v)=(2,12)$, which
return $a=\frac{1}{4}$, $\frac{625}{4}$.

\bigskip

Case $(m,k)=(6,2)$. From Lemma 29 in \cite{Sch1} we know that
$f(a,x)=h(x)h(-x)$, where as before $h(x)=x^4+px^3+qx^2+rx+s.$
Comparing coefficients in the equality $f(a,x)=h(x)h(-x)$ we
get the following system of equations
\begin{equation*}
  1+p^2-2q=0,   \; -q^2+2pr-2s=0, \; 1+r^2-2qs=0,\; a-s^2=0.
\end{equation*}
Solving the first, the second and the fourth equation with the
respect to $a, q, r$ we get
\begin{equation*}
q=\frac{p^2+1}{2},\quad r=\frac{1+2p^2+p^4+8s}{8p},\quad a=s^2.
\end{equation*}
We are thus left with finding the rational points on the genus two curve
\begin{equation*}
H_{8,6,2}:\;v^2=2(p^6+p^4-p^2-9),
\end{equation*}
where $v=\pm\left(\frac{8s-3p^4-2p^2+1}{2p}\right)$.
There is an obvious map with $p^2=X$ to the elliptic curve
\[ E_{8,6,2}:\; Y^2=2(X^3+X^2-X-9), \]
which has rank 0 and trivial torsion. Thus there are no rational points on
$H_{8,6,2}$

\bigskip

Case $(m,k)=(6,4)$. As above, $f(a,x)=h(x)h(-x)$, with $h(x)=x^4+px^3+qx^2+rx+s.$
Comparing coefficients in the equality $f(a,x)=h(x)h(-x)$ we
get the following system of equations
\begin{equation*}
  1+p^2-2q=0,   \; 1-q^2+2pr-2s=0, \; r^2-2qs=0,\; a-s^2=0.
\end{equation*}
Solving the first, the second and the fourth equation with the
respect to $a,q,r$ we get
\begin{equation*}
q=\frac{p^2+1}{2},\quad r=\frac{-3+2p^2+p^4+8s}{8p},\quad a=s^2.
\end{equation*}
We are thus left with finding rational points on the genus two curve
given by
\begin{equation*}
H_{8,6,4}:\;v^2=2(p^2+1)(p^4+3).
\end{equation*}
where $v=\pm\left(\frac{8s-3p^4-2p^2-3}{2p}\right)$.
There is an obvious map with $p^2=X$ to the elliptic curve
\[ E_{8,6,4}:\; Y^2=2(X+1)(X^2+3), \]
which has rank 0 and torsion group $\{\cal{O},(1,1),(1,-1),(0,0)\}$.
The points $(1,\pm 1)$ return $a=4$ and $a=0$.
\end{proof}

\begin{thm}
Let $f(a,x)=x^{10}+x^{6}+x^{2}+a$. Then $f(a,x)$ is reducible if and
only if $a=-u^5-u^3-u$ for some $u\in\Q$.
\end{thm}

\begin{proof}
It is clear that if $a=-u^5-u^3-u$ for some $u\in\Q$ then the
polynomial $f(a,x)$ is divisible by $x^2-u$. Now suppose that $a$
is not of the form $-u^5-u^3-u$. Then any possible factor is of
degree $\geq 3$ and from Lemma 29 in \cite{Sch1} we know that
$f(a,x)=-h(x)h(-x)$, where
\begin{equation*}
h(x)=x^5+px^4+qx^3+rx^2+sx+t.
\end{equation*}
Comparing now the coefficients in the equality $f(a,x)=-h(x)h(-x)$
we get the following system of equations
\begin{equation*}
  p^2-2q=0,   \; 1-q^2+2pr-2s=0, \; r^2-2qs+2pt=0,\; a+t^2=0, \; 1-s^2+2rt=0.
\end{equation*}
Solving the first four equations with respect to $q, r, t, a$ we get
\begin{equation*}
q=\frac{p^2}{2},\; t=\frac{p^2s-r^2}{2p},\; r=\frac{p^4-4+8s}{8p},\; a=-t^2.
\end{equation*}
On substitution, we are left with finding the rational points on:
\begin{equation*}
C_{10,6,2}:\; P^3-4(10s+3)P^2+16(12s^2+4s-29)P+64(2s-1)^3 = 0,
\end{equation*}
where $P=p^4$. The curve $C_{10,6,2}$ is of genus 1, and taking $(P,s)=(-4,1)$
as zero point, a cubic model is
\[ E_{10,6,2}:\;Y^2=X^3+X^2-24X+36.\]
This curve has rank 0 and torsion group $\Z/2\Z \times \Z/4\Z$,
with the eight points given by $\{\cal{O},(6,\pm 12),(2,0),(-6,0),(0,\pm 6),(3,0) \}.$
Thus the complete set of rational points on $C_{10,6,2}$ is
\begin{equation*}
(P,s) \in \left\{ (12/5,7/5),(-4,3),(-4,1),(-36,-1),(-4,1),(0,1/2),(12,-1) \right\}.
\end{equation*}
None of the $P$-coordinates is a fourth power, which finishes the proof.
\end{proof}

\section{Some quadratic polynomials which divide infinitely many quadrinomials $f_{n,m,k}(a,x)$}\label{section6}

It is an interesting and highly non-trivial problem as to whether a given polynomial
$h\in\Q[x]$ divides infinitely many quadrinomials. This problem was addressed by L.~Hajdu and
R.~Tijdeman in \cite{HajTij}. They prove that the polynomial $h(x)$ divides infinitely many
quadrinomials if either $h$ divides two different quadrinomials with the same sequence of exponents,
or $h$ lies in the set
\begin{equation*}
\{h\in\Q[x]:\;\exists q\in \Q[x], r\in\N_{+}: \op{deg}q \leq 3\;\mbox{and}\;h(x)|q(x^{r})\;\mbox{over}\;\Q\;\}.
\end{equation*}
In this section we are interested in finding polynomials $h\in\Z[x]$ such that $h$ divides
infinitely many quadrinomials $f_{n,m,k}(a,x)$ with $a\in\Z$ and $n,m,k\in\N$ satisfying
$n>m>k>0$. First, note that the theorem of Hajdu and Tijdeman is of little use here
because our quadrinomials are of very special form. Moreover, we shall concentrate
on polynomials $h$ of degree 2. Of course each $h$ of degree 2 divides infinitely many
(general) quadrinomials but it is unclear whether there exists even one such $h$
that divides infinitely many quadrinomials $f_{n,m,k}(a,x)$. We are interested only
in polynomials that do not divide $x^{p}\pm 1$ for any $p\in \N_{+}$. It is immediate
that if $k\equiv m\equiv n \equiv 0\bmod{p}$ then $f_{n,m,k}(-3,x)$ is divisible by $x^p-1$,
and if $k\equiv m\equiv n \equiv p\bmod{2p}$ then $f_{n,m,k}(3,x)$ is divisible by $x^p+1$.
During computer experiments, we observed that if $h(x)\in S$, where
\begin{equation*}
S=\{x^2-2x+2, x^2+2x+2, x^2+3x+3\},
\end{equation*}
then $h(x)$ divides infinitely many quadrinomials $f_{n,m,k}(a,x)$ for certain values of $a$
and specific sequences of exponents $n,m,k$. Essentially, the aim of this section is to
give a precise description of the sequences of exponents $(n,m,k)$ such that
$f_{n,m,k}(a,x)\equiv 0\bmod{h(x)}$ for $h\in S$. First, note that if
$h(x)=x^2+px+q\in\Z[x]$ divides infinitely many quadrinomials $f_{n,m,k}(a,x)$, this
immediately implies via Corollary \ref{cor1} that the (exponential) Diophantine equation
\begin{equation}\label{diveq}
A_{n}(p,q)+A_{m}(p,q)+A_{k}(p,q)=0,
\end{equation}
together with the condition $B_{n}(p,q)+B_{m}(p,q)+B_{k}(p,q)\neq 0$, has infinitely many
solutions in positive integers $n,m,k$ with $n>m>k$. We see that if $h(x)\in S$ and
$\theta_{1}, \theta_{2}$ satisfy $h(\theta_{i})=0$ then $\theta_{1}/\theta_{2}$ is a root
of unity. This is not a coincidence because if the quotient of the roots of the polynomial
$x^2+px+q=0$ is not a root of unity then the equation (\ref{diveq}) has only finitely many
solutions in positive integers $n,m,k$ with $n>m>k$. One can check that if $h\in\Z[x]$ is
a polynomial of degree 2 for which the quotient of roots is a root of unity and there is
a non-zero integer $a$ such that $h(x)|f_{n,m,k}(a,x)$, then $h\in S$. This property allows us to
compute $A_{n}(p,q)$ explicitly. We gather these computations in the following:

\begin{lem}\label{values}
We have the following equalities
\begin{itemize}
\item[(1)]  If $(p,q)=(-2,2)$ then
\begin{equation*}
A_{4n}(-2,2)=0,\; A_{4n+1}(-2,2)=(-1)^{n}2^{2n},\; A_{4n+2}(-2,2)=A_{4n+3}(-2,2)=(-1)^{n}2^{2n+1}.
\end{equation*}
\item[(2)] If $(p,q)=(2,2)$ then
\begin{equation*}
A_{4n}(2,2)=0,\; A_{4n+1}(2,2)=(-1)^{n}2^{2n},\; A_{4n+2}(2,2)=-A_{4n+3}(2,2)=(-1)^{n+1}2^{2n+1}.
\end{equation*}
\item[(3)] If $(p,q)=(3,3)$ then
\begin{equation*}
\begin{array}{ll}
A_{6n}(3,3)=0,                     & A_{6n+3}(3,3)=2(-1)^{n}3^{3n+1}, \\
A_{6n+1}(3,3)=(-1)^{n}3^{3n},      & A_{6n+4}(3,3)=(-1)^{n+1}3^{3n+2}, \\
A_{6n+2}(3,3)=(-1)^{n+1}3^{3n+1},  & A_{6n+5}(3,3)=(-1)^{n}3^{3n+2}.
\end{array}
\end{equation*}
\end{itemize}
\end{lem}
\begin{proof}
The computations are immediate and follow from expressing $A_n(p,q)$
in closed form, using the formulae at (\ref{Binet}).  For example, we find
\begin{equation*}
A_{n}(-2,2)=2^{\frac{n}{2}}\sin \left( \frac{n\pi}{4} \right).
\end{equation*}
Similarly,
\begin{equation*}
A_{n}(2,2)=(-1)^{n+1} 2^{\frac{n}{2}} \sin \left( \frac{n\pi}{4}\right), \quad A_{n}(3,3)=2\cdot 3^{\frac{n-1}{2}} \sin \left(\frac{5 \pi  n}{6}\right).
\end{equation*}
\end{proof}

It is now straightforward to characterize sequences of exponents $(n,m,k)$ and values of $a$
such that $f(a,x)\equiv 0\bmod{h(x)}$, where $h\in S$.

\begin{thm}\label{div-22}
 Let $n,m,k\in\N$ with $n>m>k$, $\gcd(n,m,k)=1$, $a\in\Z\setminus\{0\}$, and suppose that $x^{n}+x^{m}+x^{k}+a\equiv 0\bmod{h(x)}$ for $h\in S$.

\begin{itemize}
\item[(1)] If $h(x)=x^2-2x+2$ then there exists a positive integer $s$ such that $(n,m,k)=(4s+5,4s+3,4s+2)$ and $a=(-1)^s \cdot 3 \cdot 2^{2s+1}$.

\item[(2)] If $h(x)=x^2+2x+2$ then $k\equiv 0\bmod{2}$ and the following holds:
\begin{itemize}
\item[(a)] If $k=4s-2$ for some $s\in\N$ then $m=4s-1$, $n=4s+4t$ with $t\geq 0$, and
\begin{equation*}
a=2^{2s-1}(-1)^{s+t+1}(2^{2t+1}+(-1)^{t+1}).
\end{equation*}
\item[(b)] If $k=4s$ for some $s\in\N$ then $m=4t+2$, $n=4t+3$ with $t\geq s$ and
\begin{equation*}
a=2^{2s}(-1)^{t+1}(2^{2t-2s+1}+(-1)^{t-s}).
\end{equation*}
\end{itemize}

\item[(3)] If $h(x)=x^2+3x+3$ then $k\equiv 0,4\bmod{6}$ and the following holds:
\begin{itemize}
 \item[(a)] If $k=6s-2$ for some $s\in\N$ then $m=6s-1$, $n=6s+6t$, with $t\geq 0$ and
\begin{equation*}
a=3^{3s-1}(-1)^{s+t+1}(3^{3t+1}+(-1)^{t+1}).
\end{equation*}
\item[(b)] If $k=6s$ for some $s\in\N$ then $m=6t+4$, $n=6t+5$, with $t\geq s$ and
\begin{equation*}
a=3^{3s}(-1)^{t+1}(3^{3t-3s+2}+(-1)^{t-s}).
\end{equation*}
\end{itemize}
\end{itemize}
\end{thm}

\begin{proof}
The idea of the proof is straightforward, and only an outline is presented. Using the values
computed in Lemma \ref{values} we perform a case by case analysis depending on the value
of $k\bmod{4}$ in the cases $(p,q)\in\{(-2,2),(2,2)\}$, and on $k\bmod{6}$ in the case
$(p,q)=(3,3)$. We illustrate the proof in the case $(p,q)=(-2,2)$.

First, observe that for $ r \not \equiv 0 \bmod{4}$, then the power of $2$ dividing
$A_r(-2,2)$ is exactly $\lfloor\frac{r}{2}\rfloor$.
Now if $k\equiv 0\bmod{4}$, then equation (\ref{diveq}) implies that $A_n(-2,2)+A_m(-2,2)=0$,
and clearly neither $m$ nor $n$ can be divisible by $4$ (otherwise $n,m,k$ are each
divisible by $4$, contradicting coprimality). Thus
$\lfloor\frac{n}{2}\rfloor=\lfloor\frac{m}{2}\rfloor$, forcing $(n,m)=(4s+3,4s+2)$ for some
integer $s$, which cannot satisfy (\ref{diveq}).
If $k\equiv 1\bmod{2}$, then
$\lfloor\frac{n}{2}\rfloor \geq \lfloor\frac{m}{2}\rfloor > \lfloor\frac{k}{2}\rfloor$, which
leads to an impossible congruence mod $2^{\lfloor\frac{m}{2}\rfloor}$
in equation (\ref{diveq}).
If $k\equiv 2\bmod{4}$ then $k=4s+2$, say, and
$\lfloor\frac{n}{2}\rfloor \geq \lfloor\frac{m}{2}\rfloor \geq 2s+1$. Equation (\ref{diveq})
forces $\lfloor\frac{m}{2}\rfloor=2s+1$, that is,
$m=4s+3$. This gives $A_n(-2,2)+(-1)^s2^{2s+2}=0$, so that $n=4s+5$.
The value of $a$ follows from $a=-B_n(p,q)-B_m(p,q)-B_k(p,q)$, coupled with the identity
$B_{n}(p,q)=-qA_{n-1}(p,q)$.

A similar, but slightly more tedious, analysis can be performed for $(p,q)\in\{(2,2),(3,3)\}$,
and we omit the details.
\end{proof}

\begin{rem}
{\rm Note that if $h(x)=x^2\pm2x+2$, then for each $s\in\N$ the polynomial
$H(x)=2^{2s}h(x/2^{s})=x^2\pm2^{s+1}x+2^{2s+1}$ divides $f(a,x)$ for infinitely many $a$
and sequences of exponents ($n,m,k)$. A similar property holds for the polynomial
$H(x)=3^{2s}h(x/3^{s})=x^2+3^{s+1}x+3^{2s+1}$, where $h(x)=x^2+3x+3$. This observation follows
from the following fact: for any $t\in\Z$ we have $A_{n}(tp,t^2q)=t^{n-1}A_{n}(p,q)$. Moreover,
for $(p,q)\in\{(-2,2),(2,2)\}$ the only prime which divides $A_{n}(p,q)$ is equal to 2.
A similar property holds for $(p,q)=(3,3)$. Indeed, in this case the only prime dividing
$A_{n}(p,q)$ is equal to 3 provided that $n\not\equiv 3\bmod{6}$, in which case the prime
$2$ can also occur, although only to exponent $1$. All these remarks follows from
Lemma \ref{values}.
}
\end{rem}

We finish this section with the following conjecture.

\begin{conj}\label{divconj}
Let $h\in\Z[x]$ of degree $\geq 3$ and suppose that $h$ does not divide any polynomial of the
form $x^p-1$. Then there are only finitely many triples of exponents $n,m,k\in\N$ with $n>m>k$
and integers $a$ such that $f_{n,m,k}(a,x)\equiv 0\bmod{h(x)}$.
\end{conj}

\section{Numerical results, open questions and conjectures}\label{section7}

In this section we collect some numerical computations, questions and conjectures concerning
various aspects of reducibility of the quadrinomial $f(a,x)$ with $a\in\Q$. We start with
a very natural question concerning the existence of multiple roots of $f_{n,m,k}(a,x)$. Note
the example
\begin{equation*}
 x^4 + x^3 + x + 1=(x+1)^2(x^2-x+1)
\end{equation*}
which shows that there exists $a$ such that $f(a,x)$ has a double root. In fact,
$x=-1$ is a double root of $x^{m+k}+x^m+x^k+1$, where $m,k$ are odd integers, $m>k$.
A question arises as to whether one can find other $a\in\Q$ with $f_{n,m,k}(a,x)$ having
a multiple root. We have not found any examples. This motivates the following.


\begin{conj}
If $a\in\Q^{*}$ and $f_{n,m,k}(a,x)=x^n+x^m+x^k+a$ with $n>m>k$ has multiple factors,
then $a=1$ and $(n,m,k)=(d(t+u),dt,du)$ for $d\geq 1$ and odd integers $t>u\geq 1$.
\end{conj}
Although we were unable to prove this conjecture we can offer a slightly weaker result.
\begin{thm}\label{tripleroot}
There does not exist $a\in\Q^{*}$ such that the polynomial $f_{n,m,k}(a,x)$ has a root of multiplicity $\geq 3$.
\end{thm}
\noindent
In view of the following Lemma, it suffices to assume that $(n,m,k)=1$. We state the Lemma
in the very concrete form needed for our purposes, although it can clearly be stated and
proved in a more general setting.
\begin{lem}
Suppose that $f_{dn,dm,dk}(a,x)=x^{dn}+x^{dm}+x^{dk}+a$ has a root of multiplicity $N$.
Then $f_{n,m,k}(a,x)$ has a root of multiplicity $N$.
\end{lem}
\begin{proof}
Let the $n$ roots (all non-zero) of $f_{n,m,k}(a,x)=0$ in $\mathbb{C}$ be $\{r_1,...,r_n\}$.
Let $z$ be a fixed $d$-th root of unity, for example $z=e^{2i\pi/d}$.
Then the $dn$ roots of $f_{dn,dm,dk}(x)=f_{n,m,k}(x^d)=0$ are given by
\[ z^{i_1}r_1^{1/d}, \; z^{i_2} r_2^{1/d}, \; ..., \; z^{i_n} r_n^{1/d}, \quad  i_k=0,...,d-1, \]
for a fixed d-th root $r_i^{1/d}$ of $r_i$, $i=1,,,,n$. Clearly for fixed $r_i^{1/d}$,
the $d$ roots $z^{i_k} r_i^{1/d}$, $i_k=0,...,d-1$, are distinct. If we assume
$f(x^d)$ has a root of multiplicity $N$, then without loss of generality the multiple
root $\rho$ satisfies
\[ \rho = z^{i_1} r_1^{1/d}  = z^{i_2} r_2^{1/d} = ... = z^{i_N} r_N^{1/d}. \]
On raising to the $d$-th power,
\[ r_1=r_2=...=r_N ( = \rho^d), \]
so that $f(x)$ has a root of multiplicity $N$.
\end{proof}
\noindent
Without loss of generality, therefore, we may assume henceforth that $(n,m,k)=1$. \\ \\
{\it Proof of Theorem \ref{tripleroot}.}
Suppose $f_{n,m,k}(x)$ has a triple root $\theta$.
Certainly $\theta$ is a double root of the first derivative
$f_{n,m,k}'(x)=n x^{n-1}+m x^{m-1}+k x^{k-1}$, so a double root of $g(x)=n x^{n-k}+m x^{m-k}+k$.
Accordingly, $\theta$ is also a root of $g'(x)=n(n-k) x^{n-k-1}+m(m-k) x^{m-k-1}$.
This latter gives
\begin{equation}
\label{n-m}
\theta^{n-m} = -\frac{m(m-k)}{n(n-k)}.
\end{equation}
Now
\[ g(\theta)=\theta^{m-k} ( n \theta^{n-m} + m ) + k = 0, \]
so that
\[ \theta^{m-k} ( -\frac{m(m-k)}{n-k} + m ) + k = 0, \]
that is,
\begin{equation}
\label{m-k}
\theta^{m-k} = -\frac{k(n-k)}{m(n-m)}.
\end{equation}

Let $d=\gcd(n-m,m-k)$, and set $\Theta=\theta^d$. Then
\[ \Theta^{\frac{n-m}{d}} = -\frac{m(m-k)}{n(n-k)}< 0, \qquad \Theta^{\frac{m-k}{d}} = -\frac{k(n-k)}{m(n-m)}<0, \]
with $\gcd(\frac{n-m}{d}, \frac{m-k}{d})=1$. Thus $\Theta$ itself is rational, $\Theta<0$, and
\begin{equation}
\label{odd}
\frac{n-m}{d} \equiv 1 \bmod{2}, \qquad \frac{m-k}{d} \equiv 1 \bmod{2}.
\end{equation}

Raising (\ref{n-m}) to the $\frac{m-k}{d}$-th power and (\ref{m-k}) to the $\frac{n-m}{d}$-th
power gives
\begin{equation}
\label{contracted}
\left(\frac{m(m-k)}{n(n-k)}\right)^{\frac{m-k}{d}} = \left(\frac{k(n-k)}{m(n-m)} \right)^{\frac{n-m}{d}}.
\end{equation}

This can be written
\[ k^\frac{n-m}{d} (n-k)^\frac{n-k}{d} n^\frac{m-k}{d} = m^\frac{n-k}{d} (m-k)^\frac{m-k}{d} (n-m)^\frac{n-m}{d}, \]
from which it follows that if $p$ is a prime with $p \mid (n,k)$, then $p \mid m$.
Accordingly we may assume that $(n,k)=1$.
Write
\[ \frac{m(m-k)}{n(n-k)} = \frac{r}{s}, \quad (r,s)=1, \qquad \frac{k(n-k)}{m(n-m)} = \frac{u}{v}, \quad (u,v)=1. \]
From (\ref{contracted}),
\[ \frac{r}{s} = \left(\frac{a}{b}\right)^{\frac{n-m}{d}}, \qquad \frac{u}{v} = \left(\frac{a}{b}\right)^{\frac{m-k}{d}}, \qquad (a,b)=1, \]
so that
\[ r=a^{\frac{n-m}{d}}, \quad s=b^{\frac{n-m}{d}}, \qquad u=a^{\frac{m-k}{d}}, \quad v=b^{\frac{m-k}{d}}. \]
Thus we obtain for integers $t,w$:
\begin{equation}
\label{tweqs}
m(m-k) = a^{\frac{n-m}{d}} t, \; n(n-k) = b^{\frac{n-m}{d}} t, \; k(n-k) = a^{\frac{m-k}{d}} w, \; m(n-m) = b^{\frac{m-k}{d}} w.
\end{equation}
Now
\[ (b^{\frac{n-m}{d}}-a^{\frac{n-m}{d}}) t = n(n-k)-m(m-k)=(n-m)(n+m-k)>0, \]
so that $b>a\geq 1$. Further, $b^{\frac{n-m}{d}}$ divides $n(n-k)$, and  $(n,n-k)=1$, so
either $b^{\frac{n-m}{d}}$ divides $n$ or $b^{\frac{n-m}{d}}$ divides $n-k$.  \\

\noindent
{\bf Case I: } $b^{\frac{n-m}{d}} \mid n, \; b \nmid n-k$. \\ \\
Write $n = b^{\frac{n-m}{d}+f} N$, with $f \geq 0$ and $(b,N)=1$.
Then the second equation at (\ref{tweqs}) gives $t=b^{f} N (n-k)$,
and substituting into the three other equations at (\ref{tweqs}),
\[ m(m-k) = a^{\frac{n-m}{d}} b^f N (n-k), \qquad k(n-k) = a^{\frac{m-k}{d}} w, \qquad m(b^{\frac{n-m}{d}+f} N - m) = b^{\frac{m-k}{d}} w. \]
The second equation tells us $b \nmid w$; and the third equation
that $b \mid m$.
Write $m=b^g M$, $g \geq 1$, $(b,M)=1$. Then the first and third equations give
\[ b^g M (b^g M - k) = a^{\frac{n-m}{d}} b^f N (n-k), \qquad b^g M (b^{\frac{n-m}{d}+f} N - b^g M) = b^{\frac{m-k}{d}} w. \]
Comparing powers of $b$, the first equation gives $f=g$. From the second,
since $\frac{n-m}{d}+f>g$, we deduce $2g=\frac{m-k}{d}$, which contradicts (\ref{odd}). \\ \\
{\bf Case II: } $b^{\frac{n-m}{d}} \mid n-k$, \; $b \nmid n$. \\ \\
Write $n-k=b^{\frac{n-m}{d}+f} N$, with $f \geq 0$ and $(b,N)=1$.
Then the second equation at (\ref{tweqs}) gives $t=n b^f N$,
and substituting into the remaining three equations at (\ref{tweqs}),
\[ m(m-k) = a^{\frac{n-m}{d}} n b^f N, \qquad k b^{\frac{n-m}{d}+f} N = a^{\frac{m-k}{d}} w, \qquad m(n-m) = b^{\frac{m-k}{d}} w. \]
The second equation implies $w=b^{\frac{n-m}{d}+f} W$, $(b,W)=1$. The first and third
equations now give
\[ m(m-k) = a^{\frac{n-m}{d}} n b^f N, \qquad m(n-m) = b^{\frac{n-k}{d}+f} W. \]
Adding,
\[ m(n-k) = a^{\frac{n-m}{d}} n b^f N + b^{\frac{n-k}{d}+f} W. \]
Since $f< \frac{n-k}{d}+f$, the power of $b$ exactly dividing the right hand side
is equal to $f$. But the power of $b$ dividing the left hand side is at least
$\frac{n-m}{d}+f$, incompatible.
These contradictions prove that $f_{n,m,k}(x)$ cannot have a triple root.  \hfill $\square$ \\ \\
We searched for reducible quadrinomials $f_{n,m,k}(a,x)$ having every irreducible factor of degree $\geq 3$.
The search was over the range $n \leq 40$, and naive height of $a$ up to $1000$.
The following parameterized factorization came to light.

\begin{lem}\label{redquad}
Let $n>m \geq 1$. The quadrinomial $x^{3n}+x^{3m}+x^{n+m}-\frac{1}{27}$ splits into two
irreducible factors in $\Q[x]$, namely
\[ x^{3n}+x^{3m}+x^{n+m}-\frac{1}{27} = (x^n+x^m-\frac{1}{3})(x^{2n}-x^{n+m}+x^{2m}+\frac{x^n}{3}+\frac{x^m}{3}+\frac{1}{9}). \]
\end{lem}
\begin{proof}
The factorization follows from the identity $a^3+b^3+c^3-3abc=(a+b+c)(a^2+b^2+c^2-ab-ac-bc)$
with $(a,b,c)=(x^n,x^m,-1/3)$. The first factor $f_1(x)=x^n+x^m-\frac{1}{3}$ is irreducible, because
$3x^{n}f_1(1/x)$ is an Eisenstein polynomial for the prime $3$. Consider the second factor
\[ f_2(x)=x^{2n}-x^{n+m}+x^{2m}+\frac{x^n}{3}+\frac{x^m}{3}+\frac{1}{9}= (x^n+\omega x^m-\frac{\omega^2}{3})(x^n+\omega^2 x^m-\frac{\omega}{3}), \]
where $\omega^2+\omega+1=0$. It suffices to show that
$g(x)=x^n+\omega x^m-\frac{\omega^2}{3}$, $n>m\geq 1$, is irreducible
in $\Q(\omega)[x]$. Equivalently, we prove
$h(x)=-3\omega x^n g(1/x)=x^n-3\omega^2 x^{n-m}-3 \omega$ is irreducible in the ring
$\Z[\omega]$. Here, $h(x)=x^n+\omega^2 \pi^2 x^{n-m}+\omega \pi^2$, where $\pi^2=-3$,
and $\pi$ is a prime of $\Z[\omega]$. \\

Suppose that $h(x)=h_1(x)h_2(x)$, $h_i(x) \in \Z[\omega][x]$. Since $h(x) \equiv x^n \bmod{(\pi)}$,
unique factorization in the quotient ring implies
\begin{align*}
h_1(x)= & x^j+\pi (a_{j-1}x^{j-1}+a_{j-2}x^{j-2}+...+a_1 x+a_0), \\
h_2(x)= & x^{n-j}+\pi (b_{n-j-1}x^{n-j-1}+b_{n-j-2}x^{n-j-2}+...+b_1 x+b_0),
\end{align*}
where without loss of generality we may suppose $j \geq n-j$. If this inequality is strict, then
the coefficient of $x^{n-j}$ in the product $h_1(x) h_2(x)$ is equal to
\[ \pi a_0 +\pi^2 a_1 b_{n-j-1} +\pi^2 a_2 b_{n-j-2}+...+\pi^2 a_{n-j} b_0. \]
However, on comparing with coefficients of $h(x)$, this coefficient is a multiple of $\pi^2$,
so that $a_0 \equiv 0 \bmod{\pi}$,
which is impossible since $a_0 b_0=\omega$. We deduce that $j=n-j$, i.e. $n=2j$.  The coefficient of $x^j$ in the product $h_1 h_2$ is now given by
\[ \pi a_0+\pi^2 a_1 b_{j-1}+\pi^2 a_2 b_{j-2}+...+\pi^2 a_{j-1} b_1+\pi b_0. \]
As before, this coefficient is a multiple of $\pi^2$, forcing $a_0+b_0 \equiv 0 \bmod{\pi}$.
Using $a_0 b_0=\omega$, we deduce $a_0^2 \equiv -\omega \equiv -1 \bmod{\pi}$,
an impossible congruence. \\

Thus $h(x)$ is irreducible in $\Z[\omega][x]$, and $f_2(x)$ is irreducible in $\Q[x]$.
\end{proof}
This result allows construction as follows of an irreducible quadrinomial $f_{n,m,k}(a,x)$
such that $f_{n,m,k}(a,x^3)$ is reducible, and each irreducible factor has degree $\geq 3$.
\begin{cor}
Let $f(x)=x^N+x^M+x^{\frac{N+M}{3}}-\frac{1}{27}$ with $NM\not\equiv 0\pmod{3}$ and
$N+M\equiv 0\pmod{3}$. Then the polynomial $f(x)$ is irreducible over $\Q$, and $f(x^3)$
is reducible over $\Q$ with each irreducible factor of degree $\geq N$.
\end{cor}
\begin{proof}
Note that $f(x^3)=f_{3N,3M,N+M}(-1/27,x)$, so that the second assertion follows from Lemma \ref{redquad}.
Suppose that $f(x)$ is reducible over $\Q$, i.e. $f(x)=h_{1}(x)h_{2}(x)$ with
$h_{1}, h_{2}\in\Q[x]$ with $\op{deg}h_{i}>1$. Then $f(x^3)=h_{1}(x^3)h_{2}(x^3)=f_{1}(x)f_{2}(x)$,
where $f_{1},f_{2}$ are the irreducible factors of $f(x^3)$ given in the statement of
Lemma \ref{redquad}. Thus $h_{1}(x^3)=bf_{i}(x)$ for some $i\in\{1,2\}$ and $b\in\Q\setminus\{0\}$.
However, from the assumptions on $N,M$, the left hand side is invariant under the mapping
$x\mapsto \omega x$, but the right hand side is not. Thus $f(x)$ is irreducible.
\end{proof}

We list in the table below other examples discovered, and include for completeness
the previously known six examples.

\begin{equation*}
\begin{array}{c|c|c|c|c|l}
  n & m & k & a & & \mbox{factor of } f_{n,m,k}(a,x) \mbox{ of minimum degree} \\
  \hline
  6  & 4  & 2    & -16 & \dagger    & x^3 \pm 3x^2+5x \pm 4  \\
  6  & 4  & 2    & -\frac{441}{256} & & x^3 - 2x^2 + \frac{5}{2}x - \frac{21}{16}\\
  6  & 5  & 3    & -\frac{1}{9} &   & x^3 - x^2 + x - \frac{1}{3}\\
  6  & 5  & 3    & \frac{1}{16} &  & x^3 + \frac{1}{2}x + \frac{1}{4}\\
  7  & 3  & 2    & 4            &  & x^3-x^2-x+2 \\
  7  & 3  & 2    & -98 & \dagger    & x^3-x^2+2x-7  \\
  7  & 5  & 3    & \pm 8 & \ddagger     & x^3 \mp x^2-x \pm 2 \\
  8  & 3  & 2    & -\frac{225}{16} & & x^4 - 2x^3 + 4x -\frac{15}{4} \\
  8  & 4  & 2    & \frac{1}{4}    & & x^4 - 2x^3 + 2x^2 - x + \frac{1}{2}\\
  8  & 4  & 2    & \frac{625}{4}  & &  x^4 - 2x^3 + 2x^2 - 7x + \frac{25}{2} \\
  8  & 6  & 3    & -\frac{1}{4}   & & x^4 - x^3 + x - \frac{1}{2}\\
  8  & 6  & 4    & 4 & \dagger    & x^4 \pm x^3 + x^2 \pm 2x + 2\\
  9  & 4  & 1    & -\frac{13}{27} & & x^3-x^2+x-\frac{1}{3}\\
  9  & 7  & 6    & \frac{4}{729}  & & x^3 - \frac{1}{3}x + \frac{1}{9}\\
  9  & 8  & 6    & 4              & & x^3+2x^2+2x+2 \\
  9  & 8  & 6    & -\frac{1}{8}   & & x^3 - x^2 + x - \frac{1}{2}\\
  9  & 8  & 6    & \frac{729}{8}  & & x^3 - x^2 - 3x + \frac{9}{2} \\
  10 & 4  & 2    & -\frac{1}{4}   & & x^5-2x^4+2x^3-x^2+\frac{1}{2} \\
  10 & 4  & 2    & -\frac{9}{16}  & & x^5-2x^4+2x^3-2x^2+2x-\frac{3}{4}\\
  10 & 4  & 2    & -\frac{441}{4} & & x^5-2x^4+2x^3+3x^2-8x-\frac{21}{2} \\
  10 & 8  & 6    & -\frac{1}{4}   & & x^5-x^4+x^3-x^2+x-\frac{1}{2} \\
  10 & 8  & 6    & -\frac{441}{4} & & x^5-5x^4+13x^3-21x^2+21x-\frac{21}{2}\\
  12 & 5  & 1    &  \frac{7}{16}  & & x^3+x+\frac{1}{2}\\
  12 & 8  & 4    & -16 & \dagger    & x^3 \pm x^2-x \mp 2 \\
  12 & 9  & 3    & -\frac{3}{16}  & & x^3 + \frac{3}{2}, x^3+x^2-\frac{1}{2}\\
  12 & 9  & 5    & \frac{1}{16}   & & x^3 - x^2 + \frac{1}{2}\\
  12 & 9  & 7    & \frac{1}{16}   & & x^3 - x + \frac{1}{2},\quad x^3+x^2+x+\frac{1}{2}\\
  12 & 10 & 3    & \frac{1}{4}    & & x^6 - x^5 + x^3 - x^2 + \frac{1}{2}\\
  12 & 10 & 8    & \frac{1}{4}    & & x^6 - x^5 + x^4 - x^3 + x^2 - x + \frac{1}{2}\\
  12 & 11 & 6    & \frac{1}{8}    & & x^4 - x^3 + x^2 - x + \frac{1}{2}\\
  13 & 11 & 4    & 4              & & x^6-x^3-x^2+2 \\
  13 & 12 & 9    & -27            & & x^4+2x^3+3x^2+3x+3\\
  14 & 8  & 2    & -4             & & x^7 - 2x^6 + 2x^5 - x^4 + 2x^2 - 3x + 2\\
  16 & 11 & 7    & \frac{1}{16}   & & x^4 - x^3 + \frac{1}{2}\\
  17 & 10 & 8    & 4              & & x^8+x^7+x^4+2x + 2\\
  17 & 14 & 8    & -16 & \dagger    & x^5+x^3-x^2-2 \\
  20 & 14 & 8    & \frac{1}{4}    & & x^{10}-2x^9+2x^8-x^7+x^5-x^4+x^2-x+\frac{1}{2} \\
  22 & 16 & 10   & -\frac{1}{4}   & & x^{11}-2x^{10}+2x^9-x^8+x^5-x^4+x^2-x+\frac{1}{2}\\
  28 & 18 & 16   & \frac{1}{64}   & & x^{14}-4x^{13}+8x^{12}-10x^{11}+8x^{10}-3x^9-2x^8+\frac{9}{2}x^7 \\
     &    &      &                & & -4x^6+2x^5-x^3+x^2-\frac{1}{2}x+\frac{1}{8} \\ \hline
\end{array}
\end{equation*}
\[    \dagger: \mbox{Jankauskas}, \qquad \ddagger: \mbox{Walsh}   \]
%
%
%

\bigskip





We finish with a conjecture related to Walsh's second question for quadrinomials defined
over finite fields.  We expect in this case that the following strong result is true.

\begin{conj}
Let $\mathbb{F}_{q}$ be a finite field with $q=p^{e}$ elements and
fix a positive integer $M$. Then there is a constant $C$ such
that for each integer $N\geq C$ there exists $a\in\mathbb{F}_{q}$
and triple $n,m,k\in\N_{+} (n>m>k)$ with $n>N$ such that each
irreducible factor of the quadrinomial $f(a,x)$ is of degree $\geq M$.
\end{conj}


\noindent {\bf Acknowledgement}: All computations in this paper were
carried out using Magma~\cite{Mag}.

\vspace*{0.25in}

\noindent School of Mathematics and Statistical Sciences, Arizona
State University, Tempe AZ 85287-1804, USA. e-mail: bremner@asu.edu

\bigskip

\noindent Jagiellonian University, Faculty of Mathematics and Computer Science, Institute of Mathematics,
{\L}ojasiewicza 6, 30-348 Krak\'ow, Poland. email:
Maciej.Ulas@im.uj.edu.pl
\end{document}